\documentclass[psamsfont,a4paper,12pt]{amsart}

\usepackage[english]{babel}   

\usepackage{setspace}     
\usepackage{enumerate} 	
\usepackage{afterpage} 	
\usepackage[usenames,dvipsnames]{color}
\usepackage{graphicx}	

\usepackage{amsmath} 	
\usepackage{amssymb}	
\usepackage{amsfonts} 	
\usepackage{latexsym} 	
\usepackage{colonequals}  

\usepackage{amsthm} 	
\usepackage{stackrel}       
\usepackage{amscd} 	

\usepackage{tabularx}	
\usepackage{longtable}	

\usepackage{algorithm}
\usepackage{algpseudocode}

\usepackage{ucs}
\usepackage[utf8x]{inputenc}

\usepackage{url}
\usepackage{hyperref}

\usepackage{verbatim} 
\usepackage{blindtext} 

\usepackage{lipsum}

\allowdisplaybreaks

\def\p{\mathfrak{p}}

\newcommand{\C}{\mathbb{C}}

\newcommand{\R}{\mathbb{R}}

\newcommand{\Z}{\mathbb{Z}} 
\newcommand{\Q}{\mathbb{Q}}

\renewcommand{\to}{\longrightarrow}

\newtheorem{Thm}{Theorem}[section]		
\newtheorem{Lemma}[Thm]{Lemma}

\newtheorem*{thm}{Theorem}	

\theoremstyle{definition}
\newtheorem{Definition}[Thm]{Definition}

\theoremstyle{remark}
 
\newtheorem*{rmk}{Remark}

\newtheorem*{ex}{Example}

\newtheorem{ind}[]{{\rm\it Indice}}

\newcommand{\bfunc}[1]{\operatorname{\mathtt{#1}}}
\newcommand{\Gen}{\operatorname{Gen}}

\usepackage{multicol}

\usepackage{tikz}
\usetikzlibrary{arrows}

\usepackage{paralist}
\usepackage{cite}

\usepackage{breqn}   

\title{The Riemann Hypothesis for period polynomials of Hilbert modular forms}
\author[Babei]{Angelica Babei}
\author[Rolen]{Larry Rolen}
\author[Wagner]{Ian Wagner}

\begin{document}
\numberwithin{equation}{section}

\begin{abstract}
There have been a number of recent works on the theory of period polynomials and their zeros. In particular, zeros of period polynomials have been shown to satisfy a ``Riemann Hypothesis'' in both classical settings and for cohomological versions extending the classical setting to the case of higher derivatives of $L$-functions. There thus appears to be a general phenomenon behind these phenomena. In this paper, we explore further generalizations by defining a natural analogue for Hilbert modular forms. We then prove that similar Riemann Hypotheses hold in this situation as well.
\end{abstract}

\maketitle

\section{Introduction and statement of results}

One of the most useful ideas for the study of spaces of modular forms is to relate these finite dimensional spaces to spaces spanned by polynomials.  In particular, the theory of modular symbols and the Eichler-Shimura isomorphism provide a canonical cohomological theory for modular forms based on special polynomials (see \cite{PP} for an excellent summary of this theory).  To be more precise, given any cuspform $f \in S_{k}$ of weight $k$ and modular on $SL_{2}(\Z)$ the {\textit{period polynomial}} is the modular integral 
\begin{equation} \label{PP}
r_{f}(X) := \int_{0}^{i \infty} f(\tau) (\tau -X)^{k-2} d \tau.
\end{equation}
These polynomials encode deep arithmetic information.  Specifically, their coefficients are essentially the {\textit{critical $L$-values}} of $f$: 
\begin{equation} \label{l}
r_{f}(X) = - \frac{(k-2)!}{(2 \pi i)^{k-1}} \sum_{m=0}^{k-2} \frac{(2 \pi i X)^{m}}{m!} L(f, k - m -1)=
\sum_{m=0}^{k-2}i^{m+k-1} \binom{k-2}{m}  X^{m} \Lambda(f, k-m-1),
\end{equation}
where the completed $L$-function is given by
\[
\Lambda(f,s):= (2 \pi)^{-s} \Gamma(s) L(f, s).
\]
 The completed $L$-function has an analytic continuation to $\C$ and satisfies the functional equation $\Lambda(f, s) = \epsilon(f) \Lambda(f, k-s)$ with $\epsilon(f) = \pm 1$.  From this functional equation one can see that the critical values are the integer values inside the critical strip, namely $s=1, 2, \dots, k-1$.  Deep conjectures such as the Birch and Swinnerton-Dyer conjecture and the Bloch-Kato  conjecture in the case of central $L$-values and Beilinson conjecture for non-central $L$-values imply that these values contain important arithmetic information (see, e.g., \cite{BK, KZ}).  Manin also showed that the $L$-values satisfy certain rationality conditions. 

\begin{thm}[Manin \cite{Manin}]
Let $f$ be a normalized Hecke eigenform in $S_{k}$ with rational Fourier coefficients.  Then there exist $\omega_{\pm}(f) \in \R$ such that
\begin{equation*}
\Lambda(f, s)/\omega_{+}(f), \ \Lambda(f, w)/\omega_{-}(f) \in \Q
\end{equation*}
for all $s, w$ with $1 \leq s, w \leq k-1$ and $s$ even, $w$ odd.
\end{thm}
For more details about the general philosophy of the arithmetic of the periods $\omega_{\pm}(f)$ see \cite{KZ}.  

The functional equation endows the period polynomial with the relation
\[
r_{f}(X) = -i^{k} \epsilon(f) X^{\frac{k-2}{2}} r_{f} \left(- \frac{1}{X} \right).
\]
This ``self-inversive'' property shows that if $\rho$ is a zero of $r_{f}(X)$ then so is $-\frac{1}{\rho}$ and so the unit circle is a natural line of symmetry for the period polynomials just as the critical line is a natural line of symmetry for the completed $L$-function.  For this reason the stipulation that all roots of the period polynomials lie on the unit circle has been termed the {\textit{Riemann hypothesis for period polynomials}} (RHPP).  The first work on this subject is due to Conrey, Farmer, and Imamo\u{g}lu \cite{CFI}, who showed that the odd part of the period polynomial for any level $1$ Hecke eigenform, apart from five so-called ``trivial zeros", all lie on the unit circle.  Shortly thereafter, El-Guindy and Raji \cite{ER} showed that the full period polynomial for any level $1$ eigenform satisfies RHPP.  Recently, Jin, Ma, Ono, and Soundararajan \cite{JMOS} used a brilliant synthesis of analytic techniques to show that the RHPP is an even more general phenomenon.  Namely, they showed that the RHPP holds for any Hecke eigenform for any congruence subgroup $\Gamma_{0}(N)$.  Given the broad nature of these results, it is natural to ask if these are initial cases of a more general phenomenon.  In two recent papers, Diamantis and the second author \cite{DR1, DR2} have explored a generalization of the RHPP which takes into account a cohomological period polynomial attached to higher $L$-derivatives.  There they conjecture that a similar phenomenon always holds and prove some test cases of this conjecture.  Here we generalize the period polynomials in another aspect and find that the RHPP still holds true.

Specifically, we consider the generalization to any Hilbert modular eigenform of parallel weight on the full Hilbert modular group.  Some previous results have been given on the cohomology theory of Hilbert modular forms and their periods \cite{BDG, YH},  but to the best of the authors' knowledge no direct analogue of the period polynomials in this case has been written down in an explicit form in the literature (see also the second remark preceeding Theorem~\ref{M}).  In analogy with \eqref{PP} we propose 
\begin{equation} \label{HPP}
r_{f}(X) := \int_{0}^{i \infty} \cdots \int_{0}^{i \infty} f( \tau)(N(\tau) - X)^{k-2} d \tau,
\end{equation}
where $f(\tau) = f(\tau_{1}, \dots, \tau_{n})$ is a parallel weight $k$ Hilbert modular eigenform for a number field $K$ of degree $n$ on the full Hilbert modular group and $N(\tau) = \tau_{1} \cdots \tau_{n}$, $d \tau = d \tau_{1} \cdots d \tau_{n}$.  In further analogy with \eqref{l} we have
\begin{equation}
r_{f}(X) = (-1)^n (k-2)! \left(\frac{D_{K}}{(2 \pi i)} \right)^{k-1} \sum_{m=0}^{k-2} \frac{(-1)^{m(n+1)} \Gamma(k-m-1)^{n-1}}{m!} \left(\frac{(2 \pi i)^{n} X}{D_{K}} \right)^{m} L(f, k-m-1)
\end{equation}
or equivalently
\[
r_{f}(X) =  \sum_{m=0}^{k-2} (-1)^{m} i^{n(k-m-1)}\binom{k-2}{m} X^{m} \Lambda(f, k-m-1),
\]
where $D_{K}$ is the discriminant of $K$ and $L(f,s)$ and $\Lambda(f,s)$ are defined for Hilbert modular forms in equations \eqref{HL} and \eqref{HLC} respectively.
\begin{rmk}
The definition of period polynomial for a Hilbert modular form given in equation \eqref{HPP} naturally extends the elliptic modular form definition by encoding the critical $L$-values of $f$ as coefficients.  These polynomials however do not satisfy all of the period relations that the polynomials in equation \eqref{PP} satisfy.  There is another natural definition of an $n$-variable function that satisfies the corresponding period relations in the Hilbert case, but it is less clear what arithmetic information the coefficients contain in this case.
\end{rmk}   
\begin{rmk}
After writing this paper, the authors have learned that YoungJu Choie has also considered period polynomials for Hilbert modular forms in forthcoming work. 
\end{rmk}
Our main result is as follows.
\begin{Thm}[The Riemann hypothesis for period polynomials of Hilbert modular forms] \label{M}
Let $f$ be a parallel weight $k$ Hilbert modular eigenform of degree $n$ on the full Hilbert modular group.  Then all of the roots of $r_{f}(X)$ lie on the unit circle.

Moreover, as $k \to \infty$, the zeros of $r_{f}(X)$ become equidistributed on the unit circle.
\end{Thm}
\begin{rmk}
The above result seems provable as well for congruence subgroups. In particular, we included the case when the Atkin-Lehner eigenvalue $\epsilon(f)=-1$ and  we would follow the same argument up until Equation (\ref{inequality}). From there, the factor in the conductor corresponding to the level would in fact lower the bounds on the weights of forms of larger level that we need to examine. However, at the time of this project, the existing infrastructure for Hilbert modular forms over cubic fields did not cover forms of larger level.
\end{rmk}

The paper is organized as follows. In Section~\ref{PrelimSection} we discuss the basic definitions and results on polynomial roots, computer calculations, and analytic number theory required for the proofs. 
The proofs of the main results are then given in Section~\ref{Proofs}. Finally, we conclude with a discussion of examples  and ideas for future directions in Section~\ref{FinalSection}.

\section*{Acknowledgements}
The authors thank Claudia Alfes, YoungJu Choie, David Farmer, Ahmad El-Guindy, Ken Ono, Vicen\unichar{539}iu Pa\unichar{537}ol,  Wissam Raji, and Markus Schwagenscheidt for helpful comments.

\section{preliminaries}\label{PrelimSection}

\subsection{Basic definitions}
In this subsection we will review the definitions of parallel integer weight Hilbert modular forms and their $L$-functions. For more details on the general theory, we refer the reader to the survey of Bruinier in \cite{Bruinier123}. Let $K$ be a number field of degree $n$ above $\Q$. Basic Galois theory implies that  there exists $n$ different embeddings $K\hookrightarrow\C$, which we will denote by $a\mapsto a^{(j)}$ for $1\leq j\leq n$.
 We will assume from here forward that $K$ is totally real.  Define the \textit{norm} of an element by $N(x) := \prod_{j=1}^{n} x^{(j)}$ and the \textit{trace} of an element by $Tr(x) := \sum_{j=1}^{n} x^{(j)}$.  Let $\mathfrak{d}_{K}$ be the \textit{different} of $K$ so that $N(\mathfrak{d}_{K}) =:D_{K}$ is the discriminant of $K$. The general linear group $GL_{2}(K)$ embeds into $GL_{2}(\R)^{n}$ via the real embeddings of $K$.  Let $GL_2^+(K) \colonequals \{ \gamma \in GL_2(K): \det \gamma \gg 0\}$ be the subgroup of matrices with totally positive determinant.  It acts on $\mathbb{H}^{n}$ via fractional linear transformations,
 \[
\begin{pmatrix} a & b \\ c & d \end{pmatrix} \tau := \left( \frac{a\tau_{1} + b}{c\tau_{1} + d}, \frac{a^{(2)} \tau_{2} + b^{(2)}}{c^{(2)} \tau_{2} + d^{(2)}}, \dots, \frac{a^{(n)} \tau_{n} + b^{(n)}}{c^{(n)} \tau_{n} + d^{(n)}} \right),
\]
where $\tau=(\tau_{1}, \dots, \tau_{n}) \in \mathbb{H}^{n}$.  If $\mathfrak{a}$ is a fractional ideal of $K$, we define the \textit{Hilbert modular group} corresponding to $\mathfrak{a}$ as
\[
\Gamma(\mathcal{O}_{K} \oplus \mathfrak{a}) := \left\{ \begin{pmatrix} a & b \\ c & d \end{pmatrix} \in GL_{2}^+(K): a, d \in \mathcal{O}_{K}, \  b \in \mathfrak{a}^{-1}, \ c \in \mathfrak{a} \right\}.
\]
Furthermore, define $\Gamma_{K} := \Gamma( \mathcal{O}_{K} \oplus \mathcal{O}_{K}) = GL_{2}^{+}(\mathcal{O}_{K})$, which we just call the {\it full Hilbert modular group}.  For $\gamma = \begin{pmatrix} a & b \\ c & d \end{pmatrix} \in GL_{2}^+(K) \hookrightarrow GL_{2}(\R)^{n}$ and $z \in \mathbb{H}^{n}$ define the automorphic factor
\[
J(\gamma, \tau) := \det(\gamma)^{-1/2} N(c\tau + d) = \prod_{j=1}^{n} \det(\gamma_{j})^{-1/2} \left( c^{(j)} \tau_{j} + d^{(j)} \right),
\]
where $\gamma_{j} = \begin{pmatrix} a^{(j)} & b^{(j)} \\ c^{(j)} & d^{(j)} \end{pmatrix}$.  
\begin{Definition} \label{Hilbert}
A holomorphic function $f\colon \mathbb{H}^{n} \to \C$ is called a holomorphic \textit{Hilbert modular form} of parallel integer weight $k= (k,k, \dots, k) \in \Z^{n}$ for $\Gamma_{K}$ if for all $\gamma = \begin{pmatrix} a & b \\ c & d \end{pmatrix} \in \Gamma_{K}$,
\[
f(\gamma \tau) = J(\gamma, \tau)^{k} f(\tau) = \det(\gamma)^{-k/2} N(c\tau + d)^{k} f(\tau).
\]

\end{Definition} 
We denote the space of holomorphic Hilbert modular forms of weight $k$ on $\Gamma_{K}$ by $M_{k}(\Gamma_{K})$.  If $\mathcal{O}_{K}$ has a unit of negative norm then $M_{k}(\Gamma_{K}) = \{0\}$ for $k$ odd, so we will suppose that $k$ is even.  If $f \in M_{k}(\Gamma_{K})$ vanishes at the cusps we call it a cusp form and denote this space by $S_{k}(\Gamma_{K})$.  Each $f \in M_{k}(\Gamma_{K})$ has a Fourier expansion of the form
\begin{equation}
\label{expansion1}
f(\tau) =a(0) + \sum_{\substack{\nu \in \mathfrak{d}_{K}^{-1} \\ \nu \gg 0}} a(\nu) e^{2 \pi i Tr( \nu \tau)},
\end{equation}
where $Tr(\nu \tau) = \sum_{j=1}^{n} \nu^{(j)} \tau_{j}$ and $\nu \gg 0$ means that $\nu$ is totally positive. Since $\nu \in \mathfrak{d}_{K}^{-1}$, each ideal $\mathfrak{n}=\nu \mathfrak{d}_{K}$ is integral. When the forms have parallel even weight, $a(\nu)=a(\nu \eta)$ for any totally positive unit $\eta \in \mathcal{O}_K^\times$ and we may rewrite \eqref{expansion1} as 

\[
f(\tau) =a(0) + \sum_{\substack{\mathfrak{n}  \subset \mathcal{O}_K  \\ \mathfrak{n} \ne 0}} a(\mathfrak{n}) \sum_{\substack{\eta \in \mathcal{O}_{K}^{\times} \\ \eta\gg 0}}e^{2 \pi i Tr( \nu \eta \tau)},
\]
and we may identify each modular form by the coefficients $a(\mathfrak{n})$.

Therefore, $f \in S_{k}(\Gamma_{K})$ has an associated $L$-function given as a Dirichlet series by
\begin{equation} \label{HL}
L(f,s) := \sum_{\substack{ \nu \in \mathfrak{d}_{k}^{-1}/\mathcal{O}_{K}^{\times} \\ \nu \gg 0}} a(v) N(v)^{-s} = \sum_{\substack{ \mathfrak{n} \in \mathcal{O}_{K} \\ \mathfrak{n} \neq 0}} a(\mathfrak{n}) N(\mathfrak{n})^{-s}.
\end{equation}
The completed $L$-function is defined by
\begin{equation} \label{HLC}
\Lambda(f,s) := D_{K}^{s} (2 \pi)^{-ns} \Gamma(s)^{n} L(f,s)
\end{equation}
and also has the $n$-fold integral representation
\[
\Lambda(f,s) = \int_{0}^{\infty} \cdots \int_{0}^{\infty} f(iy) N(y)^{s-1} dy.
\]
This completed $L$-function satisfies the  functional equation 
\begin{equation}
\label{func}
\Lambda(f,s) = \epsilon(f) \Lambda(f, k-s),
\end{equation}
where $\epsilon(f) \in \{ \pm 1\}$.  

\subsection{Computing period polynomials}
The proofs of our main results consist of two parts. Firstly, analytic techniques are used to guarantee that the theorem eventually holds true in different aspects. Then, detailed computer calculations are used to verify the small cases. As these calculations are intensive and use newly developed code, we provide a detailed description of this procedure. 
We carry out  our computations in Magma \cite{magma}. The main ingredient for our computations consist of obtaining eigenbases for subspaces of cusp forms and creating $L$-functions for these forms. The reader can find details about constructions of $L$-functions in Magma in the handbook available online. We summarize the construction for convenience. Every $L$-function in Magma is created using the command $\bfunc{LSeries}$, which relies on  the  functional equation (\ref{func}) and the factors in the definition of the completed $L$-function such as the conductor (in our case given by the discriminant), the weight, the $\Gamma$-factor, as well as finitely many  coefficients. The coefficients can either be given at each integer, or at each prime, and then generate the Euler product.  The number of coefficients required to find $L$-values up to a given precision grows with the weight and the degree of the field.

To create $L$-series of Hilbert modular forms over quadratic fields, we use the environment for Hilbert modular forms in Magma, and a slight modification of the command $\bfunc{LSeries}$. In particular, the coefficients of the $L$-function usually come embedded in an extension of $\Q$, and  $\bfunc{LSeries}$  only computes their first complex embedding. We modify the function to allow for the other complex embeddings as well.

In the case of Hilbert modular forms over cubic fields, one cannot get all the necessary cusp forms in the pre-existing environment in Magma. Instead, we use the package \cite{hmf} which implements Fourier expansions of Hilbert modular forms. In our construction of $L$-series, we input  the Fourier coefficients at each prime, which generate the Euler product.  We describe the algorithm for finding bases of Hilbert modular forms over cubic fields in Section \ref{cubic}.

\subsection{Some basic results on $L$-functions and in the theory of self-inversive polynomials}

Our proof  requires some preliminary results on Hilbert modular $L$-functions and is a generalization of the method in \cite{JMOS}.  The completed $L$ function $\Lambda(f,s)$ extends to an entire function of order one.  Its zeros are predicted to lie on the line ${\rm{Re}}(s) =\frac{k}{2}$ , but are known to lie in the strip $\left| {\rm{Re}}(s) - \frac{k}{2} \right| < \frac{1}{2}$.  We then require the famous Hadamard factorization:
\begin{equation*}
\Lambda(f,s) = e^{A+Bs} \prod_{\rho} \left( 1 - \frac{s}{\rho} \right)e^{\frac{s}{\rho}},
\end{equation*}
where the product is over all the zeros of $\Lambda(f,s)$.  Note that if $\rho$ is a zero, then so are $\bar{\rho}$ and $k-\rho$.  Using the fact that $\Lambda(f,s)$ is real-valued on the real line and the functional equation, we obtain
\begin{equation*}
B=-\sum_{\rho} {\rm{Re}} \left(\frac{1}{\rho} \right) = - \sum_{\rho} \frac{{\rm{Re}}(\rho)}{|\rho|^2}.
\end{equation*}
Using this we have that
\begin{equation} \label{z}
\Lambda(f,s) = e^{A} \prod_{\rho \in \R} \left( 1 - \frac{s}{\rho} \right) \prod_{{\rm{Im}}(\rho)>0} \left| 1 - \frac{s}{\rho} \right|^{2}
\end{equation}
for real $s$.
This is the main ingredient for the following key result. 
\begin{Lemma} \label{ineq}
The function $\Lambda(f,s)$ is monotonically increasing for $s \geq \frac{k}{2} + \frac{1}{2}$.  Furthermore,
\begin{equation*}
0 \leq \Lambda \left(f, \frac{k}{2} \right) \leq \Lambda \left(f, \frac{k}{2} +1 \right) \leq \Lambda \left(f, \frac{k}{2} + 2 \right) \leq \dots.
\end{equation*}
If $\epsilon(f) =-1$, then $\Lambda \left(f, \frac{k}{2} \right) =0$ and 
\begin{equation*}
0 \leq \Lambda \left( f, \frac{k}{2} +1 \right) \leq \frac{1}{2} \Lambda \left(f, \frac{k}{2} + 2 \right) \leq \frac{1}{3} \Lambda \left(f, \frac{k}{2} +3 \right) \leq \dots.
\end{equation*}
\end{Lemma}
\begin{proof}
The proof follows exactly mutatis mutandis from Lemma 2.1 in \cite{JMOS}.
\end{proof}
We also require the following estimate.
\begin{Lemma} \label{L}
If $0<a<b$ and $f$ is a parallel weight $k$ newform of degree $n$, then we have
\begin{equation*} 
\frac{L \left( f, \frac{k+1}{2} +a \right)}{L \left(f, \frac{k+1}{2} +b \right)} \leq \frac{\zeta(1+a)^{2n}}{\zeta(1+b)^{2n}}.
\end{equation*}
\end{Lemma}
\begin{proof}
We have 
\begin{equation*}
-\frac{L'}{L}(f, s) =: \sum \frac{\Lambda_{f}(\mathfrak{a})}{N(\mathfrak{a})^{s}} = \sum \frac{c_{f}(m)}{m^s}.
\end{equation*}
Since $f$ is a parallel weight $k$ newform, by the Ramanujan bound \cite{DB} we have $\Lambda_{f}(\mathfrak{a}) \leq 2 N(\mathfrak{a})^{\frac{k-1}{2}} \Lambda_{K}(\mathfrak{a})$, where $\Lambda_{k}(\mathfrak{a})$ is the Von Mangoldt function for the field $K$.  We also know that if $\sum \frac{\Lambda_{K}(\mathfrak{a})}{N(\mathfrak{a})^s} = \sum \frac{c_{K}(m)}{m^s}$, then $c_{K}(m) \leq n \Lambda(m)$, where $\Lambda(m)$ is the usual Von Mangoldt function.  Thus we have
\begin{align*}
-\frac{L'}{L}(f, s) &\leq 2 \sum \frac{N(\mathfrak{a})^{\frac{k-1}{2}} \Lambda_{K}(\mathfrak{a})}{N(\mathfrak{a})^s} = \sum \frac{c_{K}(m)}{m^{s-\frac{k-1}{2}}} \\
& \leq 2n \sum \frac{\Lambda(m)}{m^{s-\frac{k-1}{2}}} = -2n \frac{\zeta'}{\zeta} \left(s - \frac{k-1}{2} \right).
\end{align*}
We then have
\begin{align*}
\frac{L \left(f, \frac{k+1}{2} +a \right)}{L \left( f, \frac{k+1}{2} +b \right)} &= \exp \left( \int_{a}^{b} - \frac{L'}{L} \left(f, \frac{k+1}{2} +t \right) dt \right) \\
& \leq \exp \left( 2n \int_{a}^{b} - \frac{\zeta'}{\zeta}(1+t) dt \right) = \frac{\zeta(1+a)^{2n}}{\zeta(1+b)^{2n}}.
\end{align*}
\end{proof}
We will also use the following theorem for determining whether a polynomial has all of its roots on the unit circle \cite{LS}.
\begin{Thm} \label{roots}
A necessary and sufficient condition for all the zeros of a polynomial $P(z) = \sum_{n=0}^{d} a_{n}z^{n}$ with complex coefficients to lie on the unit circle is that there exists a polynomial $Q(z)$, with all of its zeros inside or on the unit circle, such that
\begin{equation*}
P(z) = z^{m} Q(z) + e^{i \theta} Q^{*}(z),
\end{equation*}
where for a polynomial $g(z)$ of degree $d$, $g^{*}(z) = z^{d} \overline{g}(1/z)$.
\end{Thm}
In order to use this theorem let $m := \frac{k-2}{2}$ and define the two important polynomials $P_{f}(X)$ and $Q_{f}(X)$ by
\begin{equation}
P_{f}(X) := \frac{1}{2} \binom{2m}{m} \Lambda \left(f, \frac{k}{2} \right) + \sum_{j=1}^{m} \binom{2m}{m+j} \Lambda \left(f, \frac{k}{2} +j \right) X^{j}
\end{equation}
and
\begin{equation}
Q_{f}(X) := \frac{1}{\Lambda(f,2m+1)} P_{f}(X).
\end{equation}
We will be able to apply Theorem~\ref{roots} in our situation as a short calculation shows that
\begin{equation*}
r_{f}(i^{n+2} X) = i^{n(2m+1)} \epsilon(f) \Lambda(f, 2m+1) X^m \left[ Q_{f}(X) + \epsilon(f) Q_{f}\left( \frac{1}{X} \right) \right].
\end{equation*}

\section{Proof of the main results}\label{Proofs}

\subsection{The cases $m=1$ and $m=2$} 
The arguments here for small weights exactly mirror those in \cite{JMOS}.  For this reason we will just sketch out the proofs and refer the reader to \cite{JMOS} for more details.  For weight $k=4$ we have $m=1$ and $P_{f}(X) = \Lambda(f,2) + \Lambda(f,3)X$.  If $\epsilon(f)=-1$, then $\Lambda(f,2)=0$ so we have
\begin{equation*}
P_{f}(X) - P_{f} \left(\frac{1}{X} \right) = \Lambda(f,3) \left(X - \frac{1}{X} \right),
\end{equation*}
which clearly has roots at $X=\pm 1$.  If $\epsilon(f) =1$, then
\begin{align*}
P_{f}(X) + P_{f} \left(\frac{1}{X} \right) &= 2\Lambda(f,2) + \Lambda(f,3) \left( X + \frac{1}{X} \right) = 2 \Lambda(f,2) + 2\Lambda(f,3) \cos(\theta),
\end{align*}
where $X=e^{i \theta}$.  By Lemma \ref{ineq} we know $\Lambda(f,2) < \Lambda(f,3)$ so the equation
\begin{equation*}
\cos(\theta) = - \frac{\Lambda(f,2)}{\Lambda(f,3)}
\end{equation*}
has two solutions with $\theta \in [0, 2 \pi)$.     

For $k=6$, we have $m=2$ so 
\begin{equation*}
P_{f}(X) = 3 \Lambda(f,3) + 4 \Lambda(f,4)X + \Lambda(f,5)X^2.
\end{equation*}
If $\epsilon(f)=-1$, then $\Lambda(f,3)=0$ and we have
\begin{align*}
P_{f}(X) - P_{f} \left(\frac{1}{X} \right) &= 4 \Lambda(f,4) \left(X - \frac{1}{X} \right) + \Lambda(f,5) \left( X^2 - \frac{1}{X^2} \right) \\
&= \left(X - \frac{1}{X} \right) \left[ 4\Lambda(f,4) + \Lambda(f,5) \left( X + \frac{1}{X} \right) \right].
\end{align*}
We clearly have $X = \pm 1$ as two solutions.  By Lemma \ref{ineq} again we have $2 \Lambda(f,4)< \Lambda(f,5)$ so the two solutions to $\cos(\theta) = - \frac{2 \Lambda(f,4)}{\Lambda(f,5)}$ for $\theta \in [0, 2 \pi)$ give two other roots on the unit circle.  If $\epsilon(f)=1$, letting $X= e^{i \theta}$  we have
\begin{equation*}
P_{f}(X) + P_{f} \left(\frac{1}{X} \right) = 6 \Lambda(f,3) + 8\Lambda(f,4) \cos(\theta) + 2 \Lambda(f,5) \cos(2 \theta).
\end{equation*}
We aim to show this has two zeros with $\theta \in [0, \pi)$ and thus four zeros with $\theta \in [0, 2 \pi)$.  Noting
\begin{align*}
\frac{d}{d \theta} \left[P_{f}(e^{i \theta}) + P_{f}(e^{-i \theta}) \right] &= -8 \sin(\theta) \left( \Lambda(f,4) + \Lambda(f, 5) \cos(\theta) \right),
\end{align*}
we have critical points at $0, \pi$, and the solution $\theta_{0} \in [0, \pi)$ to $\cos(\theta) = - \frac{\Lambda(f, 4)}{\Lambda(f,5)}$.  To ensure there are two roots in $[0, \pi)$ we need $P_{f}(e^{i \theta}) + P_{f}(e^{-i \theta})$ to be positive at $\theta =0$ and $\pi$ and negative at $\theta= \theta_{0}$.  We clearly have positivity at $\theta=0$.  Positivity at $\theta = \pi$ is equivalent to
\begin{equation*}
3\Lambda(f, 3) + \Lambda(f,5) > 4 \Lambda(f,4)
\end{equation*}
while negativity at $\theta=\theta_{0}$ is equivalent to
\begin{equation*}
2 \Lambda(f, 4)^2 + \Lambda(f,5)^2 \geq  3 \Lambda(f,3) \Lambda(f,5).
\end{equation*}
By Lemma \ref{ineq} and a result of Waldspurger \cite{W} we know that $\Lambda(f,3), \Lambda(f,4)$, and $\Lambda(f,5)$ are all non-negative.  We can therefore use Lemma 4.1 in \cite{JMOS} as it is used there to prove the necessary inequalities.

\subsection{The case of large weight}
We will now prove Theorem \ref{M} for all but finitely many cases.  We will compare $Q_{f}(X)$ to $X^m$ and use Rouch\'{e}'s Theorem to show $Q_{f}(X)$ has all its zeros inside the unit circle.  Once this is established we apply Theorem \ref{roots} to complete the proof.  On $|X|=1$ we have
\begin{align}
\label{inequality}
\begin{split}
Q_{f}(z) -X^m &=\frac{1}{2} \frac{\Gamma(m+1)^{n-2}}{\Gamma(2m+1)^{n-1}} \left(\frac{(2 \pi)^{n}}{D_{K}} \right)^{m} \frac{L(f, m+1)}{L(f, 2m+1)} \\
&+ \sum_{j=1}^{m-1} \frac{1}{j!} \left( \frac{(2 \pi)^{n}}{D_{K}} \right)^{j} \left(\frac{\Gamma(2m+1-j)}{\Gamma(2m+1)} \right)^{n-1} \frac{L(f, 2m+1-j)}{L(f, 2m+1)}.
\end{split}
\end{align}
We now use Lemma \ref{L}, the fact that $\zeta(1/2)^2 \leq \frac{11}{5}$, and Minkowski's bound
\begin{equation*}
D_{K} \geq \left(\frac{n^n}{n!} \right)^2
\end{equation*}
to obtain
\begin{align*}
\left| Q_{f}(z) - X^m \right| & \leq \frac{1}{2} \frac{\Gamma(m+1)^{n-2}}{\Gamma(2m+1)^{n-1}} \left(\frac{(2 \pi)^{n}}{D_{K}} \right)^{m} \left(\frac{\zeta(1/2)}{\zeta(1/2 +m)} \right)^{2n} \\
&+ \sum_{j=1}^{m-1} \frac{1}{j!} \left( \frac{(2 \pi)^{n}}{D_{K}} \right)^{j} \left(\frac{\Gamma(2m+1-j)}{\Gamma(2m+1)} \right)^{n-1} \left(\frac{\zeta(1/2 + m-j)}{\zeta(1/2 +m)} \right)^{2n} \\ 
& \leq \frac{1}{2} \frac{\Gamma(m+1)^{n-2}}{\Gamma(2m+1)^{n-1}} \left(\frac{(2 \pi)^{n} (n!)^2}{n^{2n}} \right)^{m} \left(\frac{11}{5} \right)^{n} \\
&+ \sum_{j=1}^{m-1} \frac{1}{j!} \left( \frac{(2 \pi)^{n} (n!)^2}{n^{2n}} \right)^{j} \left(\frac{\Gamma(2m+1-j)}{\Gamma(2m+1)} \right)^{n-1} \left(\frac{\zeta(1/2 + m-j)}{\zeta(1/2 +m)} \right)^{2n} \\
&=: T_{n}(m)
\end{align*}
Therefore we need to show that $T_{n}(m) <1$ for $ n \geq 2$ and $m$ big enough.  The numbers $T_{n}(m)$ are decreasing as $n$ increases because each individual term is decreasing.  We will now show that $T_{n}(m)$ is also decreasing in $m$.  Therefore once we have $T_{2}(m_{0}) <1$ for some $m_{0}$, then we automatically have that $T_{n}(m_{0})<1$ for any $n \geq 2$ and $m \geq m_{0}$.  We will do this by showing $T_{n}(m+1) - T_{n}(m) \leq 0$. The term outside the sum in $T_{n}(m+1) - T_{n}(m)$ is
\begin{align*}
&\frac{1}{2} \frac{\Gamma(m+1)^{n-2}}{\Gamma(2m+1)^{n-1}} \left(\frac{(2 \pi)^{n} (n!)^2}{n^{2n}} \right)^{m} \left( \frac{11}{5} \right)^{n} \left[ \frac{(2 \pi)^{n} (n!)^{2}}{2^{n-1} (m+1)(2m+1)^{n-2} n^{2n}} -1 \right],
\end{align*}
which is less than or equal to zero as soon as $ m \geq 4$ for $n=2$ and is true for $m \geq 1$ for any $n \geq 3$.  Each term in the sum looks like 
\begin{align} \label{z}
\begin{split}
&\frac{1}{j!} \left( \frac{(2 \pi)^{n} (n!)^{2}}{n^{2n}} \right)^{j} \left( \frac{\Gamma(2m+1-j)}{\Gamma(2m+1)} \right)^{n-1} \left( \frac{\zeta(1/2 + m -j)}{\zeta(1/2 + m)} \right)^{2n} \\
& \times \left[ \left( \frac{(2m+2-j)(2m+1-j)}{(2m+2)(2m+1)} \right)^{n-1} \left( \frac{\zeta(1/2 + m) \zeta( 3/2 + m -j)}{\zeta(3/2 + m) \zeta(1/2 + m -j)} \right)^{2n} -1 \right].
\end{split}
\end{align}
We can use the facts that 
\begin{equation*}
\frac{1}{\zeta(3/2 + m)}, \frac{\zeta(1/2 + m)}{\zeta(1/2 + m -j)} \leq 1, \quad \zeta(3/2 + m -j)^{2} \leq \frac{8}{5} 2^{j-m} +1
\end{equation*}
to show that each term is less than or equal to zero once
\begin{equation*}
\left( \frac{(2m+2-j)(2m+1-j)}{(2m+2)(2m+1)} \right)^{n-1}  \left( \frac{8}{5} 2^{j-m} +1 \right)^{n} \leq 1.
\end{equation*}
The last term to satisfy this inequality is the $j=1$ term.  This case is equivalent to $\left( \frac{m}{m+1} \right)^{n-1} \left( \frac{16}{5} 2^{-m} + 1 \right)^{n} \leq 1$ which one can check is true once $m \geq 6$ for any $n \geq 2$.  Once we know the inequality is satisfied for $m \geq 6$, we can go back to \eqref{z} and check the remaining values of $m$ directly.  We find that equation \eqref{z} is negative for any $m \geq 1$ for $n \geq 2$.  The last thing to deal with is the fact that $T_{n}(m+1)$ has one extra factor in the sum compared to $T_{n}(m)$.  We will pair this term with the $j=m-1$ terms.  Using similar inequalities as above we must show that
\begin{align*}
&\frac{(2 \pi)^{n} (n!)^2}{m n^{2n}} \left( \frac{m+2}{(2m+2)(2m+1)} \right)^{n-1} \left(\frac{\zeta(1/2 + m)}{\zeta(3/2 +m)} \right)^{2n} \\
&+ \left(\frac{(m+3)(m+2)}{(2m+2)(2m+1)} \right)^{n-1} \left(\frac{\zeta(1/2 +m) \zeta(5/2)}{\zeta(3/2 +m) \zeta(3/2)} \right)^{2n} \leq 1,
\end{align*}
which occurs once $m \geq 3$ for $n = 2$ and $m \geq 2$ for $n \geq 3$.  We have shown that $T_{n}(m)$ is decreasing in both $n$ and $m$ so we just need find an $m_{0}$ such that $T_{2}(m_{0}) <1$.  A computer calculation shows this first occurs for $m=8$.  For higher degrees we can run this calculation again to reduce the number of cases that need to be checked explicitly.  For example $T_{3}(m)<1$ once $m \geq 5$ and $T_{n}(m)<1$ for $m \geq 3$ once $n \geq 5$.  We reduce the number of remaining cases by allowing the discriminant to vary.  For $n=2$, we have the following table that  shows the inequality is satisfied once $m$ is big enough depending on the discriminant.
\begin{center}
\begin{tabular}{ | c | c | c | c | c | c | c | c | c | c | c |   }
\hline
  $D_{K}$ & $5$ & $8$ & $12$ & $13$ & $17$ & $21$ & $24$ & $29$ & $33$ & $\geq 35$ \\ \hline
  $m \geq$ & $7$ & $6$ & $5$ & $5$ & $4$ & $4$ & $4$ & $4$ & $4$ & $3$ \\ \hline
\end{tabular}
\end{center}
Similarly, for $n=3$ the inequality is satisfied for $m \geq 3$ once we have $D_{K} \geq 84$.  The only other case we need to check is $n=4$.  The inequality is true for $m \geq 3$ once we have $D_{K} \geq 209$ and the totally real quartic field with smallest discriminant has discriminant equal to $725$.  The fact that there are not many cases to check explicitly is not too surprising after some reflection; increasing any aspect such as degree of the number field, discriminant, or weight of the form helps the polynomial satisfy the analytic conditions needed to have all its roots on the unit circle.

\subsection{Remaining cases}

We check manually the finitely many cases not covered by the previous subsection. Once we obtained the spaces of modular forms, we check that the roots are on the unit circle by testing the inequality $|Q_f(X)-X^m|<1$ as in Equation (\ref{inequality}). The inequality holds for all but $11$ polynomials associated to forms over quadratic fields. In such cases,  we check that  the trigonometric polynomials $P_f(X)+\epsilon(f)P_f\left(\frac{1}{X}\right)$ with $X=e^{i\theta}$ have the necessary number of roots on the interval $[0, \pi)$  as in \cite{JMOS}.

All the spaces for the quadratic fields are available in Magma. For small enough discriminant of the field and weight of the space, such computations can be done relatively fast on a personal computer. In the quadratic case, checking all the forms with precision of 15 decimal places took 4 hours on a 4 core Intel(R) Core(TM) i7--4720HQ CPU $@$ 2.60GHz  personal computer with 8GB of memory.

\begin{ex}
Let $K=\Q(\sqrt{5})$. For weight $k=8$, we have a unique cusp form $f$ whose period polynomial is 
\begin{align*}
r_f(X) &\approx -0.273825X^6 - 0.371966X^5 - 0.329503X^4 - 0.297572X^3 \\
&- 0.329503X^2 - 0.371966X - 0.273825,
\end{align*}
which we can write as $r_f(X) \approx  -0.273825(X^6+\frac{361}{300}X^4+\frac{361}{300}X^2+1) - 0.371966(X^5+\frac{4}{5}X^3+X)$. We  obtain that  $\Lambda(f, 6)=\frac{25}{6}\Lambda(f,4)$, as computed by Yoshida in \cite{YH}.
\end{ex}

\begin{ex} 
Let $K=\Q(\sqrt{33})$. The cusp subspace $S_8(\Gamma_K)$ has three irreducible Hecke submodules, one of which is one-dimensional. Let $g \in S_8(\Gamma_K)$ be the eigenform corresponding to this submodule. Then the period polynomial is approximately
\begin{equation*}
\begin{aligned}
r_g(X)\approx &-140158.98X^6 - 24794.709X^5 - 2025.1361X^4 
\\
&- 130.74X^3 - 2025.1361X^2 - 24794.709X- 140158.98.
\end{aligned}
\end{equation*}
\end{ex}

 The reason for the small precision in the quadratic case is due to slow computations of Hecke eigenvalues. In the cases of fields with small discriminants $D_K\le 17$ and narrow class number $h_+=1$, we were able to increase the precision by using the same technique described for creating spaces of Hilbert modular forms over cubic fields. However, this involved many tedious tests for finding generators of spaces, since the number of  generators rises quickly with the discriminant. We illustrate an example of weight $22$, which was not reachable using the existing infrastructure.

\begin{ex}
Let $K=\Q(\sqrt{5})$. Consider the eigenform $h \in S_{22}(\Gamma_K)$ with Fourier expansion in Table \ref{5x22}. The first row of the table gives totally positive generators of the first few ideals with $\omega$ a root of the polynomial $x^2-x-1$, the second row the norm of the ideal, and the third row the coefficient corresponding to the given ideal.

\begin{table}[h]
\caption{Fourier expansion of an eigenform $h \in S_{22}(\Gamma_K)$ over the field $K=\Q(\sqrt{5})$}
\label{5x22}
\begin{tabular}{|c|c|c|c|c|c|c|c|}
\hline
 $\mathfrak{n}$& (0) &(1) & $(2)$& $(\omega+2)$  &  $(3)$& $(\omega+3)$ & $(4)$ \\ \hline
$N(\mathfrak{n})$ & 0 &1 &4&5  &  $9$& $11$ & 16 \\ \hline
$a(\mathfrak{n})$&0  &1 & -4111360 &21640950 &-4319930070 &-94724929188 & 12505234538496 \\ \hline
\end{tabular}
\end{table}

The roots of the period polynomial $r_h(X)$, seen in  Figure \ref{5x22roots}, are distributed nearly uniformly on the unit circle.

\begin{figure}
\caption{}
\label{5x22roots}
\begin{center}
\includegraphics[scale=0.6]{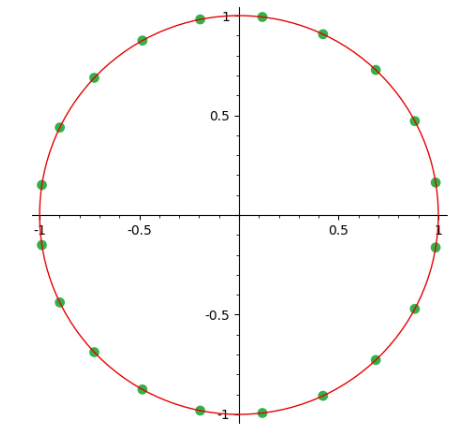} $\quad$
\end{center}
\end{figure}

\end{ex}

In the cubic case, we only need to consider the two totally real fields with discriminants $49$ and $81$, and the only special case is $m=3$. The algorithm we use to reconstruct the spaces of weight $8$ for cubic fields is outlined in Section \ref{cubic}, along with further examples. In the cubic field case, the inequality $|Q_f(X)-X^m|<1$ as in Equation (\ref{inequality}) held for all the polynomials.

\section{Examples and remarks}\label{FinalSection}

\subsection{The case of cubic fields}
\label{cubic}

We create the spaces of Hilbert modular forms using the package \cite{hmf}, which implements Fourier expansions of Hilbert modular forms, and where we can perform operations such as multiplication and applying Hecke operators. The main source of forms in the package are Eisenstein series. In general, one cannot generate full spaces of Hilbert modular forms just using Eisenstein series, but we were able to use Hecke operators on products of Eisenstein series to generate spaces of low weights $2 \le k \le 8$ for the two cubic fields we needed to investigate. In particular, we obtain Fourier expansions of forms that generate spaces of weight $k$ using Algorithm \ref{cubicalg} recursively for parallel even weights starting with $k=2$.

\begin{algorithm}
\caption{Algorithm for reconstructing full spaces of weight $k$ for a cubic field $K$}\label{cubicalg}
\begin{algorithmic}[1]
\Procedure{FullSpaceAndGenerators}{$K$}							
\State Let $d_k \colonequals \dim_\C(M_k(\Gamma)) $  (see  \cite[Addendum (3.14)]{TV}); \Comment{Actual dimension}
\State Let $E_k$ be the Eisenstein series of parallel weight $k$;
\State Compute the set $R$ of forms of weight $k$ obtained from multiplying forms  of lower weights in $\Gen_i$ for $i<k$;
\State $\Gen_k=\{E_k\} \cup R$;
\State Let $M$ be the vector space  generated by $\Gen_k$;
\Repeat                                                                                        \Comment{Keep adding new generators}
 \State Let $g\colonequals T_{\mathfrak{p}}(f)$ for primes $\mathfrak{p}$ of increasing norm; 
\State Let $V$ be the vector space generated by $\Gen_k \cup \{g\}$;
\State If $\dim_\C(V)>\dim_\C(M)$ then $\Gen_k=\Gen_k \cup \{g\}$;							
\State $M = $ the vector space generated by $\Gen_k$
\Until{$\dim_k(M) =d_k$} \Comment{until we have filled the space}
\State \textbf{return} $M, \Gen_k$;
\EndProcedure
\end{algorithmic}
\end{algorithm}

In step (2) of Algorithm \ref{cubicalg}, the dimensions are given by the following Hilbert series from \cite[Addendum (3.14)]{TV}, where the space of weight $k$ corresponds to the coefficient for $t^{k/2}$. For the cubic field with $D_K=49$, we have the series

\[\frac{(1+t^4+3t^5+5t^6+4t^7+3t^8+3t^9+3t^{10}+2t^{11}-2t^{13}+t^{14})}{(1-t)(1-t^2)(1-t^3)(1-t^7)}.\]  For this field, the spaces of weights $k=2,4,6,8$ are generated by Eisenstein series and their products, and we did not need to do the repeat loop in the algorithm.

For the cubic  field with $D_K=81$, the Hilbert series is

\[\frac{(1-t+t^2+t^3+t^4+6t^5+4t^6-2t^7+4t^8+6t^9-t^{10}+3t^{11}+3t^{12}-3t^{13}+t^{14})}{(1-t)^2(1-t^2)(1-t^9)}.\]

Besides Eisenstein series and their products, we need additional generators for weights $k=4,6$ and $8$. For weight $4$ we take $T_2(E_2^2)$, for weight $6$ we take $T_2(E_2^3)$,  and for weight $8$ we take $T_{\mathfrak{p}}(E_2^4)$ and $T_{\mathfrak{q}}(E_2^4)$, where $\p$ and $\mathfrak{q}$ lie above $17$.

 Once we have the full space, we can find the subspace of cusp forms, from which we need to extract a basis of eigenforms by finding matrices of Hecke operators. We use the ideal lying above $7$ for $D_K=49$, and an ideal above $17$ for $D_K=81$. Once we have a basis of eigenforms for each weight, we  construct the $L$-series using the required information described earlier.

\begin{ex} 
Let $K=\Q(\zeta_7+\zeta_7^{-1})$. Take the eigenform $h \in S_8(\Gamma)$ with Fourier expansion in Table \ref{49}, where $\alpha$ is a root of the polynomial $x^2 + \frac{3}{392}x - \frac{1}{21952}$.
Then the period polynomial attached to $h$, where we take the first complex embedding of $\Q(\alpha)$, is 
\[r_h(X) \approx -4.12785iX^4 + 1.29074X^3 + 0.547495iX^2 - 1.29074X- 4.12785i.\]

\begin{table}[h]
\caption{Fourier expansion of  an eigenform $h \in S_8(\Gamma)$ over the field $\Q(\zeta_7+\zeta_7^{-1})$}
\label{49}
\begin{tabular}{|c|c|c|c|c|c|c|c|}
\hline
$N(\mathfrak{n})$ & 0 &1 &7&8  &  13& 13 & 13 \\ \hline
$a(\mathfrak{n})$&$0$  &$1$ &$21952\alpha$ &$-43904\alpha-152 $&$21952\alpha-378$ & $21952\alpha-378$ & $21952\alpha-378$  \\ \hline
\end{tabular}
\end{table}
\end{ex}

\subsection{Numerical stability of the roots of the polynomials}

In this subsection, we perform experiments, first  proposed by Zagier \cite{DZ}, to examine how much leeway such polynomials have to have all the roots on the unit circle. In particular, we decompose $r_f=r_f^++r_f^-$ into the odd and even part, and check thresholds $t>0$ for which $r_f^++t \cdot r_f^-$ still has roots on the unit circle.

Zagier noticed that in the classical case for $f=\Delta$, the interval around $1$ containing $t$ was rather small, roughly $t \in [0.999964, 1.000023]$.   We investigate some classical cases for forms with larger weights and levels, as well as the Hilbert case for various weights and fields with varied discriminants. For the classical case, our experiments are summarized in Table \ref{classt}, where we take newforms with the specified level and weight. We only consider values for $t$ in the interval $[0,2]$, although values for $t$ outside this interval might work as well.

\begin{table}[h]
\caption{Values for $t$ where $r_f(X)^+ + t \cdot r_f^-(X)$ has roots on the unit circle: classical forms}
\label{classt}
\begin{tabular}{|c|c|r|}
\hline
Weight & Level& Interval for $t$ \\ \hline
12 & 1 &  $[0.999964, 1.000023]$ \\ \hline
12 & 5 &  $[0.97877, 1.02507]$ \\ \hline
12 & 7 &  $[0.9298, 1.0558]$ \\ \hline
12 & 11 &  $[0.501, 1.118]$ \\ \hline
18 & 1 &  $[0.999978, 1.000054]$ \\ \hline
18 & 5 &  $[0.9594, 1.015]$ \\ \hline
18 & 7 &  $[0.9313, 1.032]$ \\ \hline
18 & 11 &  $[0.618, 1.077]$ \\ \hline
24 & 1 &  $[0.9999871, 1.0000063]$ \\ \hline
24 & 5 &  $[0.9809, 1.0123]$ \\ \hline
24 & 7 &  $[0.9135, 1.0273]$ \\ \hline
24 & 11 &  $[0.657, 1.066]$ \\ \hline
42 & 1 &  $[0.999985, 1.000013]$ \\ \hline
100 & 1 &  $[0.999989, 1.000006]$ \\ \hline
\end{tabular}
\end{table}

We note a few observations. First, the intervals don't change too much as we vary the weight, but they do get much larger as we increase the level. They also get less centered around $1$ as we increase the level.

\begin{table}[h]
\caption{Values for $t$ where $r_f(X)^+ + t \cdot r_f^-(X)$ has roots on the unit circle, for some Hilbert modular forms}
\label{hilbt}
\begin{tabular}{|c|c|r|}
\hline
Weight & $K$ & Interval for $t$ \\ \hline
8 & $\Q(\sqrt{5})$ &  $[0,1.1158]$ \\ \hline
10 & $\Q(\sqrt{5})$ &  $[0,1.302]$ \\ \hline
12 & $\Q(\sqrt{5})$ &  $[0, 1.519]$ \\ \hline
14 & $\Q(\sqrt{5})$ &  $[0, 1.7283]$ \\ \hline
8 & $\Q(\sqrt{13})$ &  $[0,2]$ \\ \hline
8 & $\Q(\sqrt{33})$ &  $[0,2]$ \\ \hline
8 & $\Q(\zeta_7+\zeta_7^{-1})$ &  $[0,2]$ \\ \hline
\end{tabular}
\end{table}

In Table \ref{hilbt}, we investigate some cases for Hilbert modular forms, as we vary the weight $k$ and the field $K$. We note that in the Hilbert case, increases in weight do increase the interval significantly, as does the increase in the discriminant of the field.

\subsection{Questions for further research}
We conclude with a few remaining topics for future investigations.
\begin{enumerate}
\item Can a full cohomology theory be developed to explain the full context of the period polynomials defined here, for example, in relation to the above cited work of \cite{DL,YH}?
\item Is there a more general RHPP behind polynomials attached to a suitable cohomology theory?
\item Is there a Manin-type theory of these zeta-polynomials, similar to that developed in \cite{ORS}?
\end{enumerate}


\begin{thebibliography}{99}

 \bibitem{hmf}
A.  Babei, B.  Breen, S. Chari, E.  Costa, M.  Musty, S.
Schiavone, S. Sethi, S. Tripp, J. Voight, \textit{Computing canonical rings of Hilbert
modular varieties}  (2019), GitHub repository: \href{https://github.com/edgarcosta/hilbertmodularforms}{https://github.com/edgarcosta/hilbertmodularforms}

\bibitem{BDG}
M. Bertolini, H. Darmon, and P. Green, {\it Periods and points attached to quadratic algebras}, Heegner points and Rankin $L$-series, Cambridge Univ. Press, Cambridge, vol. 49 (2004), 323--367.

\bibitem{DB}
D. Blasius, {\it Hilbert modular forms and the Ramanujan conjecture}, Consani C., Marcolli M. (eds) Noncommutative Geometry and Number Theory. Aspects of Mathematics, (2006) 35--56.

\bibitem{BK}
S. Bloch and K. Kato, {\it $L$-functions and Tamagawa numbers of motives}, Grothendieck Festschrift, Vol 1(1990), Birkh\"auser, 333--400.


\bibitem{magma}
W. Bosma, J. Cannon, and C. Playoust. The Magma algebra system. I. The user language, \textit{J. Symbolic Comput.} 24 (1997), 235--265.

\bibitem{Bruinier123}
J. H. Bruinier, G. van der Geer, G. Harder, and D. Zagier, {\it The 1-2-3 of modular forms,} (K Ranestad, editor), Universitext, Springer, Berlin (2008), Lectures from the Summer School on Modular Forms and their Applications held in Nordfjordeid, June 2004.



\bibitem{CFI}
J.B. Conrey, D.W. Farmer, and \"O. Imamo\u{g}lu,{\it The nontrivial zeros of period polynomials of modular forms lie on the unit circle}, Int. Math. Res. Not. no. 20, 4758--4771 (2013).

\bibitem{DL}
H. Darmon and A. Logan, {\it Periods of Hilbert modular forms and rational points on elliptic curves}, Int. Math. Res. Not. no. 40 2153--2180 (2003).

\bibitem{DR1}
N. Diamantis and L. Rolen, {\it Eichler cohomology and zeros of polynomials associated to derivatives of $L$-functions,} accepted in Journal f\"ur die reine und angewandte Mathematik (Crelle's Journal) .

\bibitem{DR2}
N. Diamantis and L. Rolen, {\it Period polynomials, derivatives of $L$-functions, and zeros of polynomials,} Res. Math. Sci., in the collection: Modular Forms are Everywhere: Celebration of Don Zagier's 65th Birthday, {\bf 5} (9) (2018). 

\bibitem{ER}
A. El-Guindy, W. Raji, {\it Unimodularity of zeros of period polynomials of Hecke eigenforms,} Bull. Lond. Math. Soc. {\bf 46} no. 3, 528--536 (2014).

\bibitem{JMOS}
S. Jin, W. Ma, K. Ono, and K. Soundararajan, {\it The Riemann Hypothesis for period polynomials of modular forms,} Proc. Natl. Acad. of Sci. U.S.A. {\bf 113} no. 10, 2603--2608 (2016).


\bibitem{KZ} 
M. Kontsevich and D. Zagier, {\it Periods,} Mathematics unlimited --2001 and beyond, 771--808, Springer, Berlin, 2001.

\bibitem{LS}
M. Lal\'{i}n and C. Smyth, {\it Unimodularity of zeros of self-inversive polynomials} Acta Math. Hungar. 138 (2013) 85--101.

\bibitem{Manin}
 Y. T. Manin, {\it Periods of parabolic points and $p$-adic Hecke series,} Math. Sb., 371--393 (1973).
 


\bibitem{ORS}
K. Ono, L. Rolen, and F. Sprung, {\it Zeta-polynomials for modular form periods,} Adv. Math. {\bf 306}, 328--343 (2017).

\bibitem{PP}
V. Pasol and A. Popa, {\it Modular forms and period polynomials,} Proc. Lond. Math. Soc. (3) {\bf 107} (2013), no. 4, 713--743. 

\bibitem{TV}
 E. Thomas and A. T. Vasquez. Rings of Hilbert modular forms, \textit{Compositio Mathematica}, Martinus Nijhoff Publishers, vol. 48, n.2, (1983), 139-165. 


\bibitem{W}
J.-L. Waldspurger, {\it Sur les valeurs de certaines fonctions $L$-automorphes en leur centre de sym\'etrie [On the values of certain automorphic $L$-functions at the center of symmetry]}, Compositio Math. {\bf 54} (1985), 173--242.

\bibitem{YH}
H. Yoshida. Cohomology and $L$-values, \textit{Kyoto J. Math.} 52(2):369--432, 2012.


\bibitem{DZ}
D. Zagier. Personal communication to Ken Ono.


\end{thebibliography}
\end{document}